\documentclass{article}
\usepackage[affil-it]{authblk}
\usepackage[utf8]{inputenc}
\usepackage{amsmath, amsfonts, amsthm, amssymb, latexsym}
\usepackage{geometry}
\usepackage{enumerate}
\usepackage{multirow}
\usepackage{rotating}
\usepackage{graphicx,color}
\usepackage{float}
\usepackage{mathrsfs}

\newtheorem{thm}{Theorem}[section]
\newtheorem{lemma}[thm]{Lemma}

\newtheorem{define}[thm]{Definition}
\newtheorem{prop}[thm]{Proposition}
\newtheorem{cor}[thm]{Corollary}

\newtheorem*{remark}{Remark}

\title{The Calder\'on Problem For The Fractional Magnetic Operator}
\author{Li Li}
\affil{Department of Mathematics, University of Washington\\
Seattle, WA 98195, USA}
\date{}

\begin{document}

\maketitle

\noindent \textbf{ABSTRACT.}\, We introduce the fractional magnetic operator involving a magnetic potential and an electric potential. We formulate an inverse problem for the fractional magnetic operator. We determine the electric potential from the exterior partial measurements of the associated Dirichlet-to-Neumann map by using Runge approximation property.

\section{Introduction}

The study of the Calder\'on problem for local operators dates back to 1980s. We refer readers to \cite{uhlmann2014inverse} for a report on the progress made in this area.
As a variation of the classical Calder\'on problem,
the Calder\'on problem for the magnetic Schr\"odinger operator 
$$(-i\nabla+ A(x))^2+ q(x)$$
where $A(x)$ is a vector-valued magnetic potential and $q(x)$ is an electric potential has been extensively studied in the past decades. See for instance, \cite{ferreira2007determining, krupchyk2014uniqueness, nakamura1995global, sun1993inverse}.
In those articles, the authors considered the Dirichlet problem
$$(-i\nabla+ A)^2u+ qu= 0\,\,\, \text{in}\,\,\Omega,\qquad
u|_{\partial \Omega}= f$$
where $\Omega$ is a bounded domain with smooth boundary. They determined both $A$ (up to a gauge equivalence) and $q$ from the knowledge of the associated Dirichlet-to-Neumann map (DN map) 
$$\Lambda_{A, q}: f\to (\partial_\nu+ iA\cdot \nu)u_f|_{\partial \Omega}$$ 
where $u_f$ is the unique solution of the Dirichlet problem and $\nu$ is
the unit outer normal on $\partial \Omega$.

In recent years, the study of the Calder\'on problem for fractional operators has also been an active research field in mathematics. This study is motivated by problems involving anomalous diffusion and random processes with jumps in probability theory. The fractional Calder\'on problem was first introduced in \cite{ghosh2016calder} where the inverse problem for the fractional operator 
$$(-\Delta)^s+ q\qquad (0< s< 1)$$ was studied. In \cite{ghosh2016calder}, the authors considered the
exterior Dirichlet problem
$$((-\Delta)^s+ q)u= 0\,\,\, \text{in}\,\,\Omega,\qquad
u|_{\Omega_e}= f$$
where $\Omega_e:= \mathbb{R}^n\setminus\bar{\Omega}$. By using the strong uniqueness property of $(-\Delta)^s$, the authors proved the Runge approximation property of $(-\Delta)^s+ q$ and the fundamental uniqueness
result: The potential $q$ in $\Omega$ can be determined from the exterior partial measurements
of the DN map 
$$\Lambda_q: f\to (-\Delta)^su_f|_{\Omega_e}.$$
This result has been generalized. In \cite{ghosh2017calderon}, the authors considered the fractional elliptic operator
$$\mathcal{L}^s:= (-\nabla\cdot (M(x)\nabla))^s$$
where $M$ is a smooth, real symmetric matrix-valued function satisfying the uniformly elliptic condition. They formulated the Calder\'on problem for $\mathcal{L}^s+ q$ and proved the corresponding uniqueness theorem. Also see 
\cite{bhattacharyya2018inverse, cekic2018calder, covi2018inverse, ghosh2018uniqueness, ruland2019fractional}
for more results related with the fractional Calder\'on problem.

In this paper, we study the Calder\'on problem for the fractional magnetic operator $\mathcal{L}^s_A+ q$. Our operator $\mathcal{L}^s_A$ is formally defined by 
\begin{equation}\label{LsA}
\mathcal{L}^s_Au(x):= 2\lim_{\epsilon\to 0^+}
\int_{\mathbb{R}^n\setminus B_\epsilon(x)}(u(x)-e^{i(x-y)\cdot A(\frac{x+y}{2})}u(y))K(x,y)\,dy
\end{equation}
where $K(x, y)$ is a function associated with the heat kernel 
$p_t(x, y)$ defined in Subsection 2.2.
We will see $\mathcal{L}^s_A+ q$ actually generalizes $\mathcal{L}^s+ q$ later. Besides, 
$\mathcal{L}^s_A$ generalizes of the fractional magnetic Laplacian
$$(-\Delta)^s_A u(x):= c_{n,s}\lim_{\epsilon\to 0^+}
\int_{\mathbb{R}^n\setminus B_\epsilon(x)}\frac{u(x)-e^{i(x-y)\cdot A(\frac{x+y}{2})}u(y)}{|x-y|^{n+2s}}\,dy$$
introduced in \cite{d2018ground}. It was proved in \cite{squassina2016bourgain} that $(-\Delta)^s_A$ converges to the magnetic 
Laplacian $(\nabla- iA(x))^2$ as $s\to 1^-$
in an appropriate sense. Hence our problem can be viewed as a generalization of the Calder\'on problem studied in \cite{ghosh2017calderon} as well as a nonlocal analogue of the Calder\'on problem for the magnetic Schr\"odinger operator.

\begin{remark}
$(-\Delta)^s_A$ has the form of a Weyl pseudo-differential operator. In fact, we consider
$$u_x: y\to e^{i(x-y)\cdot A(\frac{x+y}{2})}u(y).$$
for each fixed $x$. Since we have the equivalent singular integral and Fourier transform definition of $(-\Delta)^s$ (see for instance, \cite{kwasnicki2017ten})
$$(-\Delta)^su(x)= c_{n,s}\lim_{\epsilon\to 0^+}
\int_{\mathbb{R}^n\setminus B_\epsilon(x)}\frac{u(x)-u(y)}{|x-y|^{n+2s}}\,dy
= \mathcal{F}^{-1}(|\xi|^{2s}\mathcal{F}u(\xi))(x)$$
then by considering the value of $(-\Delta)^su_x$ at $x$, we can formally do the computation
$$(-\Delta)^s_A u(x)= (-\Delta)^s u_x (x)= (2\pi)^{-n}\int e^{ix\cdot \xi}|\xi|^{2s}\mathcal{F}u_x(\xi)\,d\xi$$
$$=  (2\pi)^{-n}\iint e^{i(x-y)\cdot (\xi+ A(\frac{x+y}{2}))}|\xi|^{2s}u(y)\,dyd\xi
= (2\pi)^{-n}\iint e^{i(x-y)\cdot \xi}|\xi- A(\frac{x+y}{2})|^{2s}u(y)\,dyd\xi.$$
In particular, when $s= \frac{1}{2}$, the symbol $|\xi- A(x)|$ corresponds to the classical relativistic Hamiltonian for a spinless particle of zero mass under the influence of the magnetic potential $A(x)$
and $(-\Delta)^{1/2}_A$ is one of the quantized kinetic energy operators.
See for instance, \cite{ichinose2013magnetic, nagase1990weyl}.
\end{remark}

To formulate the Calder\'on problem for $\mathcal{L}^s_A+ q$, we assume $A\in L^\infty(\mathbb{R}^n)$ and $q$ is regular for $\mathcal{L}^s_A$, i.e. the associated sesqulinear form $B_{A, q}$ is coercive on $\tilde{H}^s(\Omega)\times \tilde{H}^s(\Omega)$ where $\tilde{H}^s(\Omega)$ is the closure of $C^\infty_c(\Omega)$ in the Sobolev space $H^s(\mathbb{R}^n)$ to ensure that
the exterior Dirichlet problem
$$(\mathcal{L}^s_A+ q)u= 0\,\,\, \text{in}\,\,\Omega,\qquad
u|_{\Omega_e}= g$$
has a unique solution $u_g\in H^s(\mathbb{R}^n)$ for each $g\in H^s(\mathbb{R}^n)$ and the solution operator $P_{A, q}: g\to u_g$ is bounded on $H^s(\mathbb{R}^n)$. Then we can introduce the DN map $\Lambda_{A, q}: H^s(\Omega_e)\to H^s(\Omega_e)^*$, which is formally defined by
$$\Lambda_{A, q}g:= \mathcal{L}^s_Au_g|_{\Omega_e}.$$

The following theorem is the main result in this paper.
\begin{thm}
Suppose $A\in L^\infty(\mathbb{R}^n)$ and $q_j\in L^\infty(\Omega)$ are regular for $\mathcal{L}^s_A$, 
$\Omega\cup \mathrm{supp}\,A\subset B_r(0)$ for some $r> 0$,
$W_j$ are open sets s.t. $W_j\subset \Omega_e$ and $W_j\setminus \overline{B_{3r}}(0)\neq \emptyset$ ($j= 1,2$). If 
\begin{equation}\label{IDN}
\Lambda_{A, q_1}g|_{W_2}= \Lambda_{A, q_2}g|_{W_2}
\end{equation}
for any $g\in C^\infty_c(W_1)$, then $q_1= q_2$.
\end{thm}

\begin{remark}
Here we determine the electric potential $q$ from the knowledge of DN map for a fixed magnetic potential $A$. The question whether we can determine $A, q$ simultaneously from $\Lambda_{A, q}$ is still open. The assumption $W_j\setminus \overline{B_{3r}}(0)\neq \emptyset$ looks unnatural but it is essential when we show the Runge approximation property of $\mathcal{L}^s_A+ q$ based on the strong uniqueness property of $\mathcal{L}^s$ proved in \cite{ghosh2017calderon} (see Section 5 for more details).

A closely related but different work on the fractional Calder\'on problem can be found in \cite{covi2019inverse} where the fractional gradient
$\nabla^s: H^s(\mathbb{R}^n)\to L^2(\mathbb{R}^n\times \mathbb{R}^n)$ defined by
$$\nabla^s u(x, y):= c_{n,s}(u(x)- u(y))\frac{y-x}{|y-x|^{n/2+ s+ 1}}$$
was considered. Based on the identity
$$\langle (-\Delta)^s u, v\rangle= \langle\nabla^s u, \nabla^s v\rangle,$$
the author defined the operator $(-\Delta)^s_\mathcal{A}$ by
$$\langle (-\Delta)^s_\mathcal{A} u, v\rangle:= \langle(\nabla^s+ \mathcal{A}(x, y))u, (\nabla^s+ \mathcal{A}(x, y)) v\rangle$$
for a bivariate vector-valued function $\mathcal{A}(x, y)$ and then defined the DN map $\Lambda_{\mathcal{A}, q}$ associated with the fractional operator $(-\Delta)^s_\mathcal{A}+ q$. It was proved in \cite{covi2019inverse} that $(\mathcal{A}, q)$ (up to a gauge equivalence) can be determined from the exterior partial measurements of $\Lambda_{\mathcal{A}, q}$ under some appropriate assumptions on $\mathcal{A}$ and $q$.

We remark that both the operator $\mathcal{L}^s_A+ q$ here and the operator $(-\Delta)^s_\mathcal{A}+ q$ in \cite{covi2019inverse} have their own advantages. $\mathcal{L}^s_A+ q$ generalizes a broader class of fractional operators while the uniqueness theorem for $(-\Delta)^s_\mathcal{A}+ q$ is stronger (see Theorem 1.1 and Lemma 3.8 in \cite{covi2019inverse}).
\end{remark}

The rest of this paper is organized in the following way. In Section 2, we summarize the background knowledge. In Section 3, we give the rigorous definition of $\mathcal{L}^s_A$ in bilinear form. We rigorously define the exterior Dirichlet problem and the DN map associated with $\mathcal{L}^s_A+ q$ in Section 4, prove the Runge approximation property of $\mathcal{L}^s_A+ q$ and the main theorem in Section 5.~\\

\noindent \textbf{Acknowledgements.} The author is partly supported by
National Science Foundation and would like to thank Professor Gunther Uhlmann for suggesting the problem and for helpful discussions. The author also would like to thank the unknown referee for pointing out that the conditions assumed in Theorem 1.1 and Lemma 3.3 in the original version of this paper can be weakened.

\section{Preliminaries}

Throughout this paper
\begin {itemize}
\item $n\geq 2$ denotes the space dimension and
$0< s< 1$ denotes the fractional power

\item $\Omega$ denotes a bounded Lipschitz domain and
$\Omega_e:= \mathbb{R}^n\setminus\bar{\Omega}$

\item $B_r(0)$ denotes the open ball centered at the origin with radius $r> 0$ and $\overline{B_{r}}(0)$ denotes the closure of $B_r(0)$

\item $A: \mathbb{R}^n\to \mathbb{R}^n$ denotes a real vector-valued magnetic potential

\item $q$ defined on $\Omega$ denotes an electric potential

\item $c, C, C', C_1,\cdots$ denote positive constants (which may depend on some parameters)

\item $\int\cdots\int= \int_{\mathbb{R}^n}\cdots\int_{\mathbb{R}^n}$

\item $X^*$ denotes the continuous dual space of $X$ and write
$\langle f, u\rangle= f(u)$ for $u\in X,\,f\in X^*$

\item $\mathcal{S}'(\mathbb{R}^n)$ denotes the space of temperate distributions.
\end{itemize}

\subsection{Fourier transform and Sobolev spaces}
Our notations for the Fourier transform and Sobolev spaces are
$$\mathcal{F}u(\xi)= \hat{u}(\xi):= \int e^{-ix\cdot \xi}u(x)\,dx,\qquad
H^t(\mathbb{R}^n):= \{u\in \mathcal{S}'(\mathbb{R}^n):
\int (1+\vert \xi\vert^2)^t\vert \hat{u}(\xi)\vert^2d\xi<\infty\}$$
where $t\in \mathbb{R}$. For $0< s< 1$,
one of the equivalent forms of the norm $||\cdot||_{H^s}$ is
$$||u||_{H^s}:= (||u||^2_{L^2}+ \iint\frac{|u(x)-u(y)|^2}{|x-y|^{n+2s}}\,dxdy)^{1/2}.$$
We will use the natural identification
$H^{-s}(\mathbb{R}^n)= H^s(\mathbb{R}^n)^*$.

Given an open set $U$ and a closed set $F$ in $\mathbb{R}^n$, 
$$H^t(U):= \{u|_U: u\in H^t(\mathbb{R}^n)\},\qquad 
H^t_F(\mathbb{R}^n):= \{u\in H^t(\mathbb{R}^n): \mathrm{supp}\,u\subset F\},$$
$$\tilde{H}^t(U):= 
\mathrm{the\,\,closure\,\,of}\,\, C^\infty_c(U)\,\,\mathrm{in}\,\, H^t(\mathbb{R}^n).$$
Since $\Omega$ is Lipschitz bounded, then 
$\tilde{H}^s(\Omega)= H^s_{\bar{\Omega}}(\mathbb{R}^n)$
and we have the natural identification
$$H^s(\mathbb{R}^n)/\tilde{H}^s(\Omega)= H^s(\Omega_e).$$

\subsection{Spectral theory and heat kernels}
We will not use any spectral theory in later sections. Our goal here is to provide some background knowledge of $K(x, y)$, which appears in (\ref{LsA}).

For a fixed smooth real symmetric matrix-valued function $M(x)= (a_{i,j}(x))$ satisfying the uniformly elliptic condition, i.e.
$$C^{-1}_M|\xi|^2\leq \sum_{1\leq i,j\leq n}a_{i, j}(x)\xi_i\xi_j\leq C_M|\xi|^2,\qquad x, \xi\in \mathbb{R}^n,$$
$\mathcal{L}:= -\nabla\cdot(M(x)\nabla)$ is well-defined and symmetric on $C^\infty_c(\mathbb{R}^n)$. It is known 
(see for instance, \cite{davies1989heat, grigoryan2009heat}) 
that $\mathcal{L}$ extends to be a non-negative,
self-adjoint operator in $L^2(\mathbb{R}^n)$ with the domain
$$\mathrm{Dom}(\mathcal{L})=\{u\in H^1(\mathbb{R}^n): \mathcal{L}u\in L^2(\mathbb{R}^n)\}$$
and $\mathrm{spec}\, \mathcal{L}\subset [0, \infty)$. For $t\geq 0$, we define
$$e^{-t\mathcal{L}}:= \int^\infty_0 e^{-t\lambda }dE_\lambda$$
where ${E_\lambda}$ is the spectral resolution of $\mathcal{L}$.
This is a family of bounded self-adjoint operators on $L^2(\mathbb{R}^n)$. 
It is known (see Theorem 7.13 and Theorem 7.20 in \cite{grigoryan2009heat}) that, there exists a unique symmetric heat kernel $p_t(\cdot, \cdot)$ s.t. $p_t(x, y)$ is $C^\infty$-smooth jointly in
$t>0, x, y\in \mathbb{R}^n$ and
$$(e^{-t\mathcal{L}}f)(x)= \int p_t(x,y)f(y)dy,\qquad x\in \mathbb{R}^n,\, t>0,\,
f\in L^2(\mathbb{R}^n).$$
Moreover, we have the following Gaussian bounds on $p_t(x, y)$
(see Chapter 3 in \cite{davies1989heat})
$$c_1e^{-b_1\frac{|x-y|^2}{t}}t^{-\frac{n}{2}}\leq p_t(x, y)
\leq c_2e^{-b_2\frac{|x-y|^2}{t}}t^{-\frac{n}{2}},\qquad
x, y\in \mathbb{R}^n,\, t>0.$$
Now we define
$$K(x, y):= C\int^\infty_0 p_t(x, y)\frac{dt}{t^{1+s}}.$$
By using the substitution $\alpha= \frac{|x-y|^2}{t}$, we can easily get the estimate
$$ \frac{C_1}{|x-y|^{n+2s}}\leq K(x, y)= K(y, x)\leq \frac{C_2}{|x-y|^{n+2s}},\qquad x\neq y,\quad x,y\in \mathbb{R}^n.$$

\begin{remark}
Note that if $M$ is the identity matrix, then we have
$$p_t(x, y)= \frac{1}{(4\pi t)^{\frac{n}{2}}}e^{-\frac{|x-y|^2}{4t}},
\qquad K(x,y)= \frac{c}{|x-y|^{n+2s}}.$$
\end{remark}

It has been shown in \cite{ghosh2017calderon} that
$$\mathcal{L}^s u(x)= 2\lim_{\epsilon\to 0^+}
\int_{\mathbb{R}^n\setminus B_\epsilon(x)}(u(x)-u(y))K(x,y)\,dy,\qquad
u\in H^s(\mathbb{R}^n)$$
so it is clear from (\ref{LsA}) that $\mathcal{L}^s_A= \mathcal{L}^s$
when $A= 0$.

\section{Fractional Operator $\mathcal{L}^s_A$}

Recall that in Section 1 we gave the formal pointwise definition of $\mathcal{L}^s_A$ in (\ref{LsA}). Now we do a formal computation to motivate the bilinear form definition of $\mathcal{L}^s_A$. For convenience we will write 
\begin{equation}\label{EAxy}
E_A(x,y)= e^{i(x-y)\cdot A(\frac{x+y}{2})}
\end{equation}
when necessary. Note that
\begin{equation}\label{c1}
E_A(x,y)= \overline{E_A(y,x)},\qquad |E_A(x,y)|= 1,\qquad K(x, y)= K(y, x)
\end{equation}
so formally we have
$$2\int(\int_{\mathbb{R}^n\setminus B_\epsilon(x)}(u(x)-E_A(x,y)u(y))K(x,y)\,dy)
\overline{v(x)}\,dx$$
$$=2\iint_{\{|x-y|\geq \epsilon\}}(u(x)-E_A(x,y)u(y))K(x,y)\overline{v(x)}\,dydx$$
$$=\iint_{\{|x-y|\geq \epsilon\}}(u(x)-E_A(x,y)u(y))K(x,y)\overline{v(x)}\,dydx
+ \iint_{\{|x-y|\geq \epsilon\}}(u(y)-\overline{E_A(x,y)}u(x))K(y,x)
\overline{v(y)}\,dxdy$$
$$=\iint_{\{|x-y|\geq \epsilon\}}(u(x)-E_A(x,y)u(y))K(x,y)\overline{v(x)}\,dydx$$
$$- \iint_{\{|x-y|\geq \epsilon\}}(u(x)-E_A(x,y)u(y))K(x,y)\overline{E_A(x,y)}\overline{v(y)}\,dxdy$$
$$= \iint_{\{|x-y|\geq \epsilon\}}(u(x)-E_A(x,y)u(y))(\overline{v(x)-E_A(x,y)v(y)})
K(x, y)\,dxdy.$$
Let $\epsilon\to 0^+$, then formally we have
\begin{equation}\label{slf}
\langle \mathcal{L}^s_Au, \bar{v}\rangle= \iint
(u(x)-e^{i(x-y)\cdot A(\frac{x+y}{2})}u(y))
(\overline{v(x)-e^{i(x-y)\cdot A(\frac{x+y}{2})}v(y)})K(x,y)\,dxdy.
\end{equation}

\begin{define}
We define $\mathcal{L}^s_A$ by the bilinear form
\begin{equation}\label{blf}
\langle \mathcal{L}^s_Au, v\rangle:= \iint
(u(x)-e^{i(x-y)\cdot A(\frac{x+y}{2})}u(y))
(v(x)-e^{-i(x-y)\cdot A(\frac{x+y}{2})}v(y))K(x, y)\,dxdy.   
\end{equation}
\end{define}

It is clear from (\ref{blf}) that 
\begin{equation}\label{ns}
\langle \mathcal{L}^s_Au, v\rangle= \langle \mathcal{L}^s_{-A}v, u\rangle.
\end{equation}

\begin{remark} 
Note that by (\ref{EAxy}), (\ref{c1}) and (\ref{blf}), 
we have 
$$\langle \mathcal{L}^s_Au, v\rangle- \langle \mathcal{L}^s_Av, u\rangle
= \iint(u(x)- E_A(x, y)u(y))(v(x)- \overline{E_A(x, y)}v(y))K(x, y)\,dxdy$$
$$-\iint(v(x)- E_A(x, y)v(y))(u(x)- \overline{E_A(x, y)}u(y))
K(x, y)\,dxdy$$
$$= \iint(E_A(x, y)- \overline{E_A(x, y)})u(x)v(y)K(x, y)\,dxdy
- \iint(E_A(x, y)- \overline{E_A(x, y)})u(y)v(x)K(x, y)\,dxdy$$
$$= \iint(E_A(x, y)- \overline{E_A(x, y)})u(x)v(y)K(x, y)\,dxdy
- \iint(E_A(y, x)- \overline{E_A(y, x)})u(x)v(y)K(y, x)\,dydx$$
$$= 2\iint(E_A(x, y)- \overline{E_A(x, y)})u(x)v(y)K(x, y)\,dxdy$$
so in general 
$$\langle \mathcal{L}^s_Au, v\rangle\neq \langle \mathcal{L}^s_Av, u\rangle.$$ 
\end{remark}

We claim that (\ref{blf}) is a bounded bilinear form on $H^s(\mathbb{R}^n)\times H^s(\mathbb{R}^n)$ for $A\in L^\infty(\mathbb{R}^n)$. To show this, we need to consider the norm $||\cdot||_{H^s_A}$ introduced in \cite{d2018ground, squassina2016bourgain}.

\begin{define}
The magnetic Sobolev norm $||\cdot||_{H^s_A}$ is defined by
$$||u||_{H^s_A}:= (||u||^2_{L^2}+ [u]^2_{H^s_A})^{1/2}$$
where 
\begin{equation}\label{semi}
[u]_{H^s_A}:= (\iint
\frac{|u(x)-e^{i(x-y)\cdot A(\frac{x+y}{2})}u(y)|^2}{|x-y|^{n+2s}}\,dxdy)^{1/2}.
\end{equation}
\end{define}

Clearly, $||\cdot||_{H^s_A}= ||\cdot||_{H^s}$ when $A= 0$. In fact, we can show the equivalence between $||\cdot||_{H^s_A}$ and the classical $H^s$ norm for $A\in L^\infty(\mathbb{R}^n)$. The key estimate we will use is that
$$|e^{i(x-y)\cdot A(\frac{x+y}{2})}- 1|\leq C\max\{1,\,|x-y|\}$$
where $C$ depends on $||A||_{L^\infty}$.
\begin{lemma}
Suppose $0< s< 1$ and $A\in L^\infty(\mathbb{R}^n)$, then $||\cdot||_{H^s_A}\sim ||\cdot||_{H^s}$. 
\end{lemma}
\begin{proof}
We only need to show that
\begin{equation}\label{semine}
|[u]_{H^s}- [u]_{H^s_A}|\leq C'||u||_{L^2}.
\end{equation}
In fact, by using the identity
$$|a|^2- |b|^2= (a- b)\bar{a}+ (\bar{a}- \bar{b})b$$
and (\ref{EAxy}), we have
$$|[u]^2_{H^s}- [u]^2_{H^s_A}|= 
|\iint\frac{|u(x)-u(y)|^2-|u(x)-e^{i(x-y)\cdot A(\frac{x+y}{2})}u(y)|^2}{|x-y|^{n+2s}}\,dxdy|= |I_1+ I_2|$$
where 
$$I_1:= \iint\frac{(E_A(x,y)-1)u(y)
\overline{u(x)-u(y)}}{|x-y|^{n+2s}}\,dxdy$$
$$I_2:= \iint\frac{\overline{(E_A(x, y)-1)u(y)}(u(x)- E_A(x, y)u(y))}{|x-y|^{n+2s}}\,dxdy$$
By Cauchy-Schwarz inequality we have
$$|I_1|\leq (\iint\frac{|(E_A(x,y)-1)u(y)|^2}
{|x-y|^{n+2s}}\,dxdy)^{1/2}
(\iint\frac{|u(x)-u(y)|^2}{|x-y|^{n+2s}}\,dxdy)^{1/2}$$
$$= ((\iint_{\{|x-y|\leq 1\}}+ \iint_{\{|x-y|\geq 1\}})\frac{|(E_A(x,y)-1)u(y)|^2}
{|x-y|^{n+2s}}\,dxdy)^{1/2}[u]_{H^s}$$
$$\leq(\iint_{\{|x-y|\leq 1\}}\frac{C^2|x-y|^2|u(y)|^2}{|x-y|^{n+2s}}\,dxdy
+ \iint_{\{|x-y|\geq 1\}}\frac{C^2|u(y)|^2}{|x-y|^{n+2s}}\,dxdy)^{1/2}[u]_{H^s}$$
$$=
(\int(\int_{\{|x-y|\leq 1\}}\frac{C^2}{|x-y|^{n+2s-2}}\,dx)|u(y)|^2\,dy
+ \int(\int_{\{|x-y|\geq 1\}}\frac{C^2}{|x-y|^{n+2s}}\,dx)|u(y)|^2\,dy)^{1/2}[u]_{H^s}$$
$$\leq C'||u||_{L^2}[u]_{H^s}.$$
Similarly we can show
$$|I_2|\leq C'||u||_{L^2}[u]_{H^s_A}.$$
Hence, we have
$$|[u]^2_{H^s}- [u]^2_{H^s_A}|\leq C'||u||_{L^2}([u]_{H^s}+ [u]_{H^s_A}),$$
which implies (\ref{semine}).
\end{proof}

The boundness of $\mathcal{L}^s_A$ is now an immediate consequence of the lemma above.
\begin{prop}
Suppose $0< s< 1$ and $A\in L^\infty(\mathbb{R}^n)$, then 
$$\mathcal{L}^s_A: H^s(\mathbb{R}^n)\to H^{-s}(\mathbb{R}^n)$$
is linear and bounded.
\end{prop}
\begin{proof}
Since $K(x, y)\sim 1/|x-y|^{n+2s}$, then by (\ref{slf}) and Cauchy-Schwarz inequality we have
$$|\langle \mathcal{L}^s_Au, \bar{v}\rangle|\leq C[u]_{H^s_A}[v]_{H^s_A}\leq
C'||u||_{H^s}||v||_{H^s}= C'||u||_{H^s}||\bar{v}||_{H^s}.$$
\end{proof}

\section{Exterior Dirichlet Problem and DN Map}

From now on we always assume $A\in L^\infty(\mathbb{R}^n)$ and $q\in L^\infty(\Omega)$.
\begin{define}
The sesquilinear form associated with $\mathcal{L}^s_A+ q$ is defined by
\begin{equation}\label{BAq}
B_{A, q}(u, v):= \langle \mathcal{L}^s_Au, \bar{v}\rangle+ \int_\Omega qu\bar{v},\qquad u,v\in H^s(\mathbb{R}^n).
\end{equation}
\end{define}

The boundness of $B_{A,q}$ follows from the the boundness of $\mathcal{L}^s_A$.

\begin{define}
We say $q$ is regular for $\mathcal{L}^s_A$
if $B_{A,q}$ is coercive on $\tilde{H}^s(\Omega)\times \tilde{H}^s(\Omega)$.
\end{define}

(\ref{ns}) implies $q$ is regular for $\mathcal{L}^s_A$ if and only if 
$q$ is regular for $\mathcal{L}^s_{-A}$. Now we give a sufficient condition for $q$ being regular.

\begin{prop}
Suppose $c\leq q\in L^\infty(\Omega)$ for some $c>0$, then $B_{A,q}$ is coercive on $\tilde{H}^s(\Omega)\times \tilde{H}^s(\Omega)$.
\end{prop}
\begin{proof}
Since $K(x, y)\sim 1/|x-y|^{n+2s}$, then by (\ref{slf}), (\ref{semi}) and Lemma 3.3 we have
$$B_{A, q}(u, u)\geq C[u]^2_{H^s_A}+ c||u||^2_{L^2}\geq C'||u||^2_{H^s}$$
for $u\in \tilde{H}^s(\Omega)$, so the coercivity holds.
\end{proof}

\subsection{Exterior Dirichlet problem}
\begin{define}
We say $u\in H^s(\mathbb{R}^n)$ is a weak solution of the exterior Dirichlet problem
\begin{equation}
\left\{
\begin{aligned}
(\mathcal{L}^s_A+ q)u&= f\quad \text{in}\,\,\Omega\\
u&= g\quad \text{in}\,\,\Omega_e\\
\end{aligned}
\right.
\end{equation}
where $f\in (\tilde{H}^s(\Omega))^*$ and $g\in H^s(\mathbb{R}^n)$ if $u$ satisfies
$u- g\in \tilde{H}^s(\Omega)$ and
$$B_{A, q}(u, \phi)= f(\bar{\phi}),\qquad \phi\in \tilde{H}^s(\Omega).$$
\end{define}

\begin{prop}
Suppose $q$ is regular for $\mathcal{L}^s_A$, then for each $g\in H^s(\mathbb{R}^n)$, the problem
\begin{equation}\label{EDP}
\left\{
\begin{aligned}
(\mathcal{L}^s_A+ q)u&= 0\quad \text{in}\,\,\Omega\\
u&= g\quad \text{in}\,\,\Omega_e\\
\end{aligned}
\right.
\end{equation}
has a unique solution $u_g\in H^s(\mathbb{R}^n)$ and the solution operator $P_{A, q}: g\to u_g$ is bounded on $H^s(\mathbb{R}^n)$.
\end{prop}
\begin{proof}
By Lax-Milgram Theorem, there exists an invertible bounded linear map $f\to w_f$
from $(\tilde{H}^s(\Omega))^*$ to $\tilde{H}^s(\Omega)$ s.t. $w_f$ satisfies
$$B_{A, q}(w, \phi)= f(\bar{\phi}),\qquad \phi\in \tilde{H}^s(\Omega).$$
For any fixed $g\in H^s(\mathbb{R}^n)$, let $f= -(\mathcal{L}^s_A+ q)g$, then $u_g:= w_f+g$ is the unique weak solution of (\ref{EDP}) and the boundness of $P_{A, q}$ on $H^s(\mathbb{R}^n)$ is clear.
\end{proof}

\subsection{DN map}
From now on we always assume $q$ is regular for $\mathcal{L}^s_A$.

Let $X:= H^s(\mathbb{R}^n)/\tilde{H}^s(\Omega)= H^s(\Omega_e)$ 
and $\tilde{g}:=$ the natural image of 
$g\in H^s(\mathbb{R}^n)$ in $X$.

\begin{define}
We define the Dirichlet-to-Neumann map $\Lambda_{A, q}$ by 
\begin{equation}\label{DN}
\langle \Lambda_{A, q}\tilde{g}, \tilde{h}\rangle:= B_{A, q}(u_g, \bar{h}),
\qquad g,h\in H^s(\mathbb{R}^n)
\end{equation}
where $u_g= P_{A,q}g$. 
\end{define}

Note that if $g_2- g_1\in \tilde{H}^s(\Omega)$ and $h_2- h_1\in \tilde{H}^s(\Omega)$,
then $u_{g_1}= u_{g_2}$ and
$$B_{A, q}(u_{g_2}, \bar{h}_2)- B_{A, q}(u_{g_1}, \bar{h}_1)=
B_{A, q}(u_{g_2}-u_{g_1}, \bar{h}_2)+ B_{A, q}(u_{g_1}, \overline{h_2-h_1})= 0$$
so $\Lambda_{A, q}$ is well-defined.
If $g, h$ belong to the orthogonal complement of $\tilde{H}^s(\Omega)$ in $H^s(\mathbb{R}^n)$, then
$$|\langle \Lambda_{A, q}\tilde{g}, \tilde{h}\rangle|\leq
C||u_g||_{H^s}||h||_{H^s}
\leq C'||g||_{H^s}||h||_{H^s}= C'||\tilde{g}||_X||\tilde{h}||_X$$
so $\Lambda_{A, q}: X\to X^*$ is bounded.

For convenience, we just write 
$\Lambda_{A, q}g$ and $\langle \Lambda_{A, q}g, h\rangle$
instead of
$\Lambda_{A, q}\tilde{g}$ and 
$\langle \Lambda_{A, q}\tilde{g}, \tilde{h}\rangle$.

\begin{remark}
Roughly speaking, $\Lambda_{A, q}g= \mathcal{L}^s_Au_g|_{\Omega_e}$ since
we can formally do the computation
$$\langle \Lambda_{A, q}g, h\rangle=
\int (\mathcal{L}^s_Au_g)h+ \int_\Omega qu_gh$$
$$=(\int_{\Omega_e}+ \int_\Omega) (\mathcal{L}^s_Au_g)h+ \int_\Omega qu_gh
=\int_{\Omega_e} (\mathcal{L}^s_Au_g)h.$$
\end{remark}

The following integral identity will be used in Section 5 to prove the main theorem.
\begin{prop}
Suppose $q_j$ are regular for $\mathcal{L}^s_A$
($j= 1, 2$). For $g_1, g_2\in H^s(\mathbb{R}^n)$,
let $u^+_1:= P_{A, q_1}(g_1)$ and $u^-_2:= P_{-A, q_2}(g_2)$,
i.e.
$u^+_1$ is the unique weak solution of
\begin{equation}
\left\{
\begin{aligned}
(\mathcal{L}^s_{A}+ q_1)u&= 0\quad \text{in}\,\,\Omega\\
u&= g_1\quad \text{in}\,\,\Omega_e\\
\end{aligned}
\right.
\end{equation}
and $u^-_2$ is the unique weak solution of
\begin{equation}
\left\{
\begin{aligned}
(\mathcal{L}^s_{-A}+ q_2)u&= 0\quad \text{in}\,\,\Omega\\
u&= g_2\quad \text{in}\,\,\Omega_e,\\
\end{aligned}
\right.
\end{equation}
then we have
\begin{equation}\label{Ii}
\langle (\Lambda_{A, q_1}-\Lambda_{A, q_2})g_1, g_2\rangle=
\int_\Omega(q_1-q_2)u^+_1u^-_2.    
\end{equation}
\end{prop}
\begin{proof}
By (\ref{ns}), (\ref{BAq}) and (\ref{DN}), we have
$$\langle \Lambda_{A, q}g, h\rangle= B_{A, q}(P_{A,q}g, \overline{P_{-A,q}h})
= B_{-A, q}(P_{-A,q}h, \overline{P_{A,q}g})= 
\langle \Lambda_{-A, q}h, g\rangle.$$
Thus we have
$$\langle (\Lambda_{A, q_1}-\Lambda_{A, q_2})g_1, g_2\rangle=
\langle \Lambda_{A, q_1}g_1, g_2\rangle
-\langle \Lambda_{-A, q_2}g_2, g_1\rangle=
B_{A, q_1}(u^+_1, \overline{u^-_2})- B_{-A, q_2}(u^-_2, \overline{u^+_1})$$
$$=\langle \mathcal{L}^s_Au^+_1, u^-_2\rangle- 
\langle \mathcal{L}^s_{-A}u^-_2, u^+_1\rangle+ \int_\Omega(q_1-q_2)u^+_1u^-_2
= \int_\Omega(q_1-q_2)u^+_1u^-_2.$$
\end{proof}

\section{Proof of the Main Theorem}
The proof of Theorem 1.1 relies on the Runge approximation property of $\mathcal{L}^s_A+ q$, which is based on the following strong uniqueness property.

\begin{prop}
(Theorem 1.2 in \cite{ghosh2017calderon}) 
Suppose $0< s< 1$ and $u\in H^s(\mathbb{R}^n)$.
If both $u$ and $\mathcal{L}^s u$ vanish in a nonempty open set $W$, 
then $u= 0$ in $\mathbb{R}^n$.
\end{prop}

The next lemma is the bridge from the strong uniqueness property of  $\mathcal{L}^s$ to the Runge approximation property of $\mathcal{L}^s_A+ q$.

\begin{lemma}
Suppose $\Omega\cup \mathrm{supp}\,A\subset B_r(0)$ for some $r> 0$, $W$ is a nonempty open set s.t. $W\cap B_{3r}(0)= \emptyset$, then we have
$$\mathcal{L}^s u|_W= \mathcal{L}^s_Au|_W,\qquad u\in \tilde{H}^s(\Omega).$$
\end{lemma}
\begin{proof}
Let $u\in C^\infty_c(\Omega)$ and $v\in C^\infty_c(W)$.
By (\ref{EAxy}), (\ref{c1}) and (\ref{blf}), we have
$$\langle (\mathcal{L}^s- \mathcal{L}^s_A)u, v\rangle$$
$$= \iint[(u(x)- u(y))(v(x)- v(y))-
(u(x)- E_A(x, y)u(y))(v(x)- \overline{E_A(x, y)}v(y))]K(x, y)\,dxdy$$
$$= \iint(E_A(x,y)-1)u(y)v(x)K(x,y)\,dxdy+
\iint(\overline{E_A(x,y)}-1)u(x)v(y)K(x,y)\,dxdy$$
$$= \iint(E_A(x,y)-1)u(y)v(x)K(x,y)\,dxdy+
\iint(\overline{E_A(y,x)}-1)u(y)v(x)K(y,x)\,dydx$$
\begin{equation}\label{zeroint}
= 2\iint(E_A(x,y)-1)u(y)v(x)K(x,y)\,dxdy.
\end{equation}
Note that if $x\notin W$, then $v(x)= 0$; if $y\notin \Omega$, then $u(y)= 0$; if $x\in W$ and $y\in \Omega$, then 
$$|\frac{x+y}{2}|\geq \frac{|x|-|y|}{2}\geq \frac{3r-r}{2}= r,$$
which implies $E_A(x, y)= 1$ in this case.
Hence the integrand in (\ref{zeroint}) is always zero.
\end{proof}

\begin{cor}
Suppose $\Omega\cup \mathrm{supp}\,A\subset B_r(0)$ for some $r> 0$, $W$ is an open set s.t. $W\setminus \overline{B_{3r}}(0)\neq \emptyset$. If
$$u\in \tilde{H}^s(\Omega),\qquad \mathcal{L}^s_{A}u|_W= 0$$
then $u= 0$ in $\mathbb{R}^n$.
\end{cor}
\begin{proof}
By Lemma 5.2, 
$u= \mathcal{L}^s u= 0$ in $W\setminus \overline{B_{3r}}(0)$
so $u= 0$ in $\mathbb{R}^n$ by Proposition 5.1.
\end{proof}

Now we can prove the Runge approximation property of $\mathcal{L}^s_A+ q$.

\begin{prop}
Suppose $\Omega\cup \mathrm{supp}\,A\subset B_r(0)$ for some $r> 0$, $W$ is an open set s.t. $W\subset \Omega_e$ and $W\setminus \overline{B_{3r}}(0)\neq \emptyset$, then
$$S:= \{P_{A, q}f|_\Omega: f\in C^\infty_c(W)\}$$
is dense in $L^2(\Omega)$ where $P_{A, q}$ is the solution operator defined in Subsection 4.1.
\end{prop}
\begin{proof}
By the Hahn-Banach Theorem, it suffices to show that:

If $v\in L^2(\Omega)$ and $\int_\Omega vw= 0$ for all $w\in S$, then $v= 0$ in $\Omega$.

In fact, for any given $v\in L^2(\Omega)$, 
let $\phi$ be the unique weak solution of
\begin{equation}
\left\{
\begin{aligned}
(\mathcal{L}^s_{-A}+ q)\phi&= v\quad \text{in}\,\,\Omega\\
\phi&= 0\quad \text{in}\,\,\Omega_e,\\
\end{aligned}
\right.
\end{equation}
then for any $f\in C^\infty_c(W)$, we have
$$\int_\Omega vP_{A,q}f= \langle v, P_{A,q}f-f\rangle= 
\langle (\mathcal{L}^s_{-A}+ q)\phi, P_{A,q}f-f\rangle= 
\langle (\mathcal{L}^s_A+ q)(P_{A,q}f-f), \phi\rangle$$
since $P_{A,q}f- f\in \tilde{H}^s(\Omega)$. Also note that 
$$\langle (\mathcal{L}^s_A+ q)P_{A,q}f, \phi\rangle= 0$$
since $P_{A,q}$ is the solution operator and $\phi\in \tilde{H}^s(\Omega)$, so we have 
$$\int_\Omega vP_{A,q}f= -\langle (\mathcal{L}^s_A+ q)f, \phi\rangle=
-\langle \mathcal{L}^s_Af, \phi\rangle= -\langle \mathcal{L}^s_{-A}\phi, f\rangle.$$
Hence, if $v\in L^2(\Omega)$ and $\int_\Omega vw= 0$ for all $w\in S$, then
the corresponding $\phi$ satisfies
$$\phi\in \tilde{H}^s(\Omega),\qquad \mathcal{L}^s_{-A}\phi|_W= 0.$$
This implies $\phi= 0$ in $\mathbb{R}^n$ by Corollary 5.3 and thus $v= 0$ in $\Omega$.
\end{proof}

Now we are ready to prove Theorem 1.1.

\begin{proof}
For any fixed $\epsilon >0$ and $f\in L^2(\Omega)$, by Proposition 5.4 we can choose
$u^+_1= P_{A, q_1}(g_1)$ for some $g_1\in C^\infty_c(W_1)$ s.t.
$$||u^+_1-f||_{L^2(\Omega)}\leq \epsilon$$
and for this chosen $u^+_1$, we can choose $u^-_2= P_{-A, q_2}(g_2)$ 
for some $g_2\in C^\infty_c(W_2)$ s.t.
$$||u^+_1||_{L^2(\Omega)}||u^-_2-1||_{L^2(\Omega)}\leq \epsilon.$$
Now by (\ref{IDN}) and (\ref{Ii}), we have
$$\int_\Omega(q_1-q_2)u^+_1u^-_2= 0$$
so 
$$|\int_\Omega(q_1-q_2)f|= |\int_\Omega(q_1-q_2)(f-u^+_1)+ \int_\Omega(q_1-q_2)u^+_1(1-u^-_2)|
\leq C\epsilon.$$
Let $\epsilon\to 0^+$, then we have
$$\int_\Omega(q_1-q_2)f= 0.$$
Since $f\in L^2(\Omega)$ is arbitrary, then we can conclude that $q_1= q_2$.
\end{proof}

\bibliographystyle{plain}
\bibliography{Reference}

\begin{thebibliography}{10}

\bibitem{bhattacharyya2018inverse}
Sombuddha Bhattacharyya, Tuhin Ghosh, and Gunther Uhlmann.
\newblock Inverse problem for fractional-laplacian with lower order non-local
  perturbations.
\newblock {\em arXiv preprint arXiv:1810.03567}, 2018.

\bibitem{cekic2018calder}
Mihajlo Ceki{\'c}, Yi-Hsuan Lin, and Angkana R\"uland.
\newblock The {Calder\'on} problem for the fractional {Schr\"odinger} equation
  with drift.
\newblock {\em arXiv preprint arXiv:1810.04211}, 2018.

\bibitem{covi2019inverse}
Giovanni Covi.
\newblock An inverse problem for the fractional schr{\"o}dinger equation in a
  magnetic field.
\newblock {\em Inverse Problems}, 2019.

\bibitem{covi2018inverse}
Giovanni Covi.
\newblock Inverse problems for a fractional conductivity equation.
\newblock {\em Nonlinear Analysis}, 2019.

\bibitem{d2018ground}
Pietro d’Avenia and Marco Squassina.
\newblock Ground states for fractional magnetic operators.
\newblock {\em ESAIM: Control, Optimisation and Calculus of Variations},
  24(1):1--24, 2018.

\bibitem{davies1989heat}
Edward~Brian Davies.
\newblock {\em Heat kernels and spectral theory}, volume~92.
\newblock Cambridge University Press, 1989.

\bibitem{ferreira2007determining}
David Dos~Santos Ferreira, Carlos~E Kenig, Johannes Sj\"ostrand, and Gunther
  Uhlmann.
\newblock Determining a magnetic {Schr\"odinger} operator from partial cauchy
  data.
\newblock {\em Communications in mathematical physics}, 271(2):467--488, 2007.

\bibitem{ghosh2017calderon}
Tuhin Ghosh, Yi-Hsuan Lin, and Jingni Xiao.
\newblock The {Calder\'on} problem for variable coefficients nonlocal elliptic
  operators.
\newblock {\em Communications in Partial Differential Equations},
  42(12):1923--1961, 2017.

\bibitem{ghosh2018uniqueness}
Tuhin Ghosh, Angkana R\"uland, Mikko Salo, and Gunther Uhlmann.
\newblock Uniqueness and reconstruction for the fractional {Calder\'on} problem
  with a single measurement.
\newblock {\em arXiv preprint arXiv:1801.04449}, 2018.

\bibitem{ghosh2016calder}
Tuhin Ghosh, Mikko Salo, and Gunther Uhlmann.
\newblock The {Calder\'on} problem for the fractional {Schr\"odinger} equation.
\newblock {\em (to appear) Analysis and Partial Differential Equations}, 2018.

\bibitem{grigoryan2009heat}
Alexander Grigoryan.
\newblock {\em Heat kernel and analysis on manifolds}, volume~47.
\newblock American Mathematical Soc., 2009.

\bibitem{ichinose2013magnetic}
Takashi Ichinose.
\newblock Magnetic relativistic {Schr\"odinger} operators and imaginary-time
  path integrals.
\newblock In {\em Mathematical physics, spectral theory and stochastic
  analysis}, pages 247--297. Springer, 2013.

\bibitem{krupchyk2014uniqueness}
Katsiaryna Krupchyk and Gunther Uhlmann.
\newblock Uniqueness in an inverse boundary problem for a magnetic
  {Schr\"odinger} operator with a bounded magnetic potential.
\newblock {\em Communications in Mathematical Physics}, 327(3):993--1009, 2014.

\bibitem{kwasnicki2017ten}
Mateusz Kwa\'snicki.
\newblock Ten equivalent definitions of the fractional {Laplace} operator.
\newblock {\em Fractional Calculus and Applied Analysis}, 20(1):7--51, 2017.

\bibitem{nagase1990weyl}
Michihiro Nagase and Tomio Umeda.
\newblock Weyl quantized {Hamiltonians} of relativistic spinless particles in
  magnetic fields.
\newblock {\em Journal of functional analysis}, 92(1):136--154, 1990.

\bibitem{nakamura1995global}
Gen Nakamura, Ziqi Sun, and Gunther Uhlmann.
\newblock Global identifiability for an inverse problem for the {Schr\"odinger}
  equation in a magnetic field.
\newblock {\em Mathematische Annalen}, 303(1):377--388, 1995.

\bibitem{ruland2019fractional}
Angkana R\"uland and Mikko Salo.
\newblock The fractional {Calder\'on} problem: low regularity and stability.
\newblock {\em Nonlinear Analysis}, 2019.

\bibitem{squassina2016bourgain}
Marco Squassina and Bruno Volzone.
\newblock Bourgain-{Br\'ezis}-{Mironescu} formula for magnetic operators.
\newblock {\em Comptes Rendus Mathematique}, 354(8):825--831, 2016.

\bibitem{sun1993inverse}
Zi~Qi Sun.
\newblock An inverse boundary value problem for {Schr\"odinger} operators with
  vector potentials.
\newblock {\em Transactions of the American Mathematical Society},
  338(2):953--969, 1993.

\bibitem{uhlmann2014inverse}
Gunther Uhlmann.
\newblock Inverse problems: seeing the unseen.
\newblock {\em Bulletin of Mathematical Sciences}, 4(2):209--279, 2014.

\end{thebibliography}
\end{document}